\documentclass[11pt]{amsart}

\usepackage{amsmath}
\usepackage{amssymb}
\usepackage{graphicx}
\usepackage{xcolor}

\newtheorem{theorem}{Theorem}[section]
\newtheorem{lemma}[theorem]{Lemma}
\newtheorem{corollary}[theorem]{Corollary}

\theoremstyle{definition}
\newtheorem{definition}[theorem]{Definition}
\newtheorem{example}[theorem]{Example}

\theoremstyle{remark}
\newtheorem{remark}[theorem]{Remark}

\numberwithin{equation}{section}

\def\V{\Vert}

\keywords{Fixed points, actions of semigroups, metric spaces, uniform Lipschitzian mappings, Lifschitz constant, uniform normal structure, orbit-nonexpansive mappings, orbit Lipschitzian actions.}
\subjclass[2010]{47H09, 47H10, 47H20, 54E40}
\begin{document}

\title[Orbital Lipschitzian actions]{Orbital Lipschitzian mappings and semigroup actions on metric spaces}
\author{Rafael Esp\'{\i}nola,  Maria  Jap\'on, Daniel Souza}

\dedicatory{Dedicated to the memory of Professors K. Goebel and W. A. Kirk for all the moments, teachings and love they gave us.}

\address{R. Esp\'{\i}nola\hfill\break
Departamento de An\'alisis Matemático, Universidad de
Sevilla, Spain}
 \email{espinola@us.es}

 \address{M. Jap\'on \hfill\break
Departamento de An\'alisis Matemático,  Universidad de
Sevilla, Spain}
 \email{japon@us.es}

 \address{D. Souza \hfill\break
Departamento de An\'alisis Matemático,  Universidad de
Sevilla, Spain}
 \email{dsouzaufrj@gmail.com}
 
\thanks{The authors were partially supported by Ministerio de Ciencia e Innovaci\'on, Grant PGC2018-098474-B-C21, and Junta de Andaluc\'{\i}a Grants P20-00637, US-1380969 and FQM-127.}

\begin{abstract}
In this paper we study some results on common fixed points of families of mappings on metric spaces by imposing orbit  Lipschitzian conditions on them. These orbit Lipschitzian conditions are weaker than asking the mappings to be Lipschitzian in the traditional way. We provide new results under the two classic approaches in the theory of fixed points for uniformly Lipschitzian mappings: the one under the normal structure property of the space (which can be regarded as the Cassini-Maluta's approach) and the one after the Lifschitz characteristic of the metric space (Lifschitz's approach). Although we focus on the case of semigroup of mappings, our results are new even when a mapping is considered by itself.
\end{abstract}

\maketitle

\section{Introduction}

Given a metric space $(X,d)$, a mapping $T:X\to X$ is said to be uniformly $k$-Lipschitzian if there exists some $k>0$ such that
$$
d(T^n x,T^ny)\le k d(x,y)
$$
for all $x,y\in X$ and for all $n\in\mathbb{N}$.
Uniformly Lipschitzian mappings were introduced by K. Goebel and W. A. Kirk in 1973 in their seminal work \cite{GK1} (see also Chapter 16 in \cite{GK}). As a result, they proposed a generalization of nonexpansive mappings (which is the case when $k=1$ in the above definition) and for which fixed point results were still possible under adequate geometrical conditions on the space. Since then, many authors have worked on this class of mappings providing a very fruitful line of research for decades (see, for instance, \cite{CM, DKS, DuTa, Go, HH, Ishi,  L, LX, TaXu, W, WW} and references therein).

Inspired by the notion of orbit-nonexpansivity  \cite{AFH, EJS,N}, in this paper we introduce some  orbit uniformly Lipschitzian conditions which we use to  revisit the two main approaches for obtaining fixed points for uniformly Lipschitzian mappings: Lifschitz's and Casini-Maluta's approaches. In our setting, the Lipschitz condition is not evaluated for every pair of points $x$ and $y$ in $X$ in terms of their distance, but we consider orbital assumptions instead (see Section 3 for definitions), which result in weaker requirements for  the mappings.  The operators here considered are very far from being continuous as it is shown with several examples along the work. Additionally, we will consider  semigroups acting over a metric space  with the purpose of finding a common fixed point for all the operators involved. Lifschitz and Casini-Maluta Theorems, and many o\-thers in the literature, will be now obtained as particular cases of the results.

Our work is organized as follows: in the Preliminaries section, we provide some details on the history of fixed points for uniformly Lipschitzian mappings as well as we introduce the notation and definitions that will be used along the paper. Additionally, Lifschitz and Casini-Maluta approaches  will be fully stated. In Section $3$ we introduce actions by semigroups over a metric space considering the case of a single mapping as a particular example. We define orbit Lipschitzian actions as the natural extension of orbit nonexpansivity for uniformly Lipschitzian mappings. Also, a strong version of this notion is introduced that will be of very much importance for us. Nevertheless, we show that both notions coincide in a wide class of semigroups. In our two main sections, Sections $4$ and $5$, we focus on extending Lifschitz and Casini-Maluta Theorems, respectively, in metric spaces to provide common fixed points for semigroups of mappings satisfying our orbit Lipschitzian conditions. 
We also display examples that illustrate the applicability of our Lipschitz condition on the orbits as well as the optimality of our results. 



\section{Preliminaries}

We start this section by recalling some well-known definitions and  fixed point theorems for uniformly Lipschitzian mappings. Throughout all the article, by $X$ and $B(x,r)$ we will denote an arbitrary metric space and the closed ball centered at $x\in X$ with radius $r\in (0,+\infty)$, respectively. 


The starting point for this work are uniformly Lipschitzian mappings, which have already been introduced in the previous section. In the seminal paper \cite{GK1}, K. Goebel and W. A. Kirk proved a fixed point theorem for uniformly $k$-Lipschitzian mappings in the framework of uniformly convex Banach spaces for which $k(1-\delta(1/k))<1$, where $\delta(\cdot)$ denotes the modulus of uniform convexity of $X$. For the particular case of a Hilbert space, Goebel-Kirk result provides fixed point for uniformly $k$-Lipschitzian mappings with $k<5^{1/2}/2$ (see \cite[Chapter 16]{GK}). A sharper estimation for $k$, still providing the existence of fixed points for uniformly $k$-Lipschitzian mappings in Hilbert spaces,  was obtained by Lifschitz in \cite{L}. By using the geometric characteristic $\kappa(X)$ introduced in the same work, Lifschitz proved as a consequence of its main result that every uniformly $k$-Lipschitzian self-mapping defined on closed convex bounded subset of a Hilbert space has a fixed point whenever $k<2^{1/2}$. 
In general,  the best constant $k$  in Hilbert spaces for which  the existence of fixed points for  uniformly $k$-Lipschitzian mappings can be assured is unknown, although this constant has to be less that $\pi/2$ by Baillon's example \cite[Example 16.1]{GK}). 

Lifschitz Theorem was proved in the broader context of general metric spaces by using the notion of regular balls that we recall next.

\begin{definition}\label{c}
Balls in a metric space $X$ are said to be $c$-regular, for $c\ge 1$, if the following holds: For any $k<c$ there are numbers $\mu,\alpha\in (0,1)$ such that for any $x,y\in X$ and $r>0$ with $d(x,y)\ge (1-\mu)r$, there exists a $z\in X$ such that
$$
B(x, (1+\mu)r)\cap B(y, k(1+\mu)r)\subset B(z,\alpha r).
$$
The Lifschitz characteristic of $X$ is finally defined as 
$$
\kappa(X)=\sup\{c\ge 1: \mbox{ the balls in $X$ are $c$-regular}\}.
$$

\end{definition} 
\begin{remark}
	It happens that $\kappa (X)\in [1,2]$ for any metric space $X$. In particular, extreme values are reached as $\kappa (X)=1$ for $X=({\mathbb R}^2,\|\cdot\|_{\rm max})$ and $\kappa (X)=2$ for $X$ the real line with the usual metric. For a Hilbert space $X$, $\kappa (X)=2^{1/2}$.  In the framework of Banach spaces, it holds that $\kappa(X)>1$ if and only if $\epsilon_0(X)$, the characteristic of convexity of $X$, is smaller than $1$ (see, for instance, \cite[Theorem 5]{DT}). For more on Lifschitz constant, check \cite{DKS, DT} or \cite[Chapter 16]{GK}. In particular, 
	if $X$ is a complete  CAT$(0)$ space, then $\kappa(X)\ge \sqrt{2}$ with $\kappa(X)=2$  if $X$ is  an $\mathbb{R}$-tree  (see \cite[Theorem 5]{DKS}). Remember that $\mathbb{R}$-trees are a particular instance of CAT$(0)$ spaces.
\end{remark}

The following result is the main one in \cite{L} and it is the one regarded as Lifschitz Theorem (see also \cite[Theorem 16.2]{GK}).
\begin{theorem} \label{Li}
Let $X$ be a bounded complete metric space and $T:X\to X$ a  uniformly $k$-Lipschitzian mapping. If $k<\kappa(X)$, then $T$ has a fixed point.
\end{theorem} 

For a nonempty subset $C$ of  a metric space $X$, we recall the following notation:  
\begin{flalign*}
D(x, C)= &\sup\{ d(x,y)\; :\; y\in C\}, \;\; x\in X;\\
r(C)=& \inf\{ D(x,C)\; :\; x\in C\};\\
\delta(C)=& \sup\{ d(x,y)\; :\; x,y\in C\};\\
\end{flalign*}
Note that $D(x,C)$ is also denoted by $r_x(C)$ in some other references. We prefer $D(x,C)$  since it makes definitions and proofs simpler (see  \cite{EJS0}).

If $X$ is a Banach space, the normal structure coefficient $N(X)$ is  defined as the supremum of the quotients $r(C)/ \delta(C)$ when $C$ is 
a closed convex bounded subset of $X$ with positive diameter. It is said that $X$ 
 has uniform normal structure when $N(X)<1$. 
It is well known, for instance, that $N(X)<1$ implies that the Banach space is reflexive. Now we can state Casini-Maluta Theorem which was originally proved in \cite{CM}.

\begin{theorem}\label{normal} Let $X$ be a Banach space with uniform normal structure. Then, for every closed convex bounded subset $C$ of $X$, any $T:C\to C$ which is uniform $k$-Lipschitzian with $k<{N}(X)^{-1/2}$ has a fixed point. 
\end{theorem}

 In their paper, E. Casini and E. Maluta showed some examples of Banach spaces where their approach by using the coefficient $N(X)$ provides better estimations for obtaining fixed points than the Lifschitz's characteristic $\kappa(X)$.  Therefore, Lifschitz Theorem and Casini-Maluta Theorem are not related in general, both providing a different spectrum when it comes to their possible applications. 

  Casini-Maluta Theorem was later extended to metric spaces by T. C. Lim and H. K. Xu in \cite{LX}.
In order to do that, the notion of uniform normal structure was redefined replacing the family of closed convex bounded subsets by  the family of admissible subsets of a metric space $X$.  Recall that a subset $A$ of $X$ is said to be {\it admissible} if $A$ is a nonempty intersection of closed balls of $X$. 
We denote by $\mathcal{A}(X)$ the family of all {\it admissible subsets} of $X$. Admissible sets are a natural substitute for convex sets in the lack of linear structure of the underlying space. In fact, every admissible set in a Banach space is additionally closed convex and bounded. 
Also, for a nonempty subset $C$ of  a metric space $X$, we denote
\begin{flalign*}
{\rm cov} (C)= & \cap\{ B\; :\; B \text{ is a closed ball and $C\subseteq B$}\},
\end{flalign*}
where ${\rm cov} (C)$ is known as the {\it admissible cover or envelope} of $C$ in $X$. For a general metric space $X$, the {\it (metric)  normal structure coefficient $\tilde{N}(X)$} is defined as follows: 
$$
\tilde{N}(X):=\sup\left\lbrace\frac{r(A)}{\delta(A)}\right\rbrace ,
$$ 
where the supremum is taken over all admissible subsets $A$ of $X$ with positive diameter. The metric space $X$ is said to have {\it (metric) uniform normal structure} when
 $\tilde{ N}(X)< 1$.

In a similar way as reflexivity was implied by uniform normal structure in Banach spaces, uniform normal structure implies that the family of admissible sets $\mathcal{A}(X)$ is compact \cite[Theorem 5.4]{KK}. This means that 
the intersection of any collection of elements of $\mathcal{A}(X)$ is nonempty provided that   it satisfies the finite intersection property.

T. C. Lim and H. K. Xu would still require an extra assumption for their counterpart of Casini-Maluta Theorem: the so-called property (P) for sequences.

\begin{definition}\label{Psequences}
	A metric space $X$ is said to have property (P) for sequences if given any two bounded sequences $\{ x_n\}$ and $\{ z_n\}$ in $X$, such that $z_n\in  {\rm cov}(\{x_m: m\ge n\})$, there is $z\in \bigcap_n {\rm cov}(\{ z_m\; :\; m\geq n\})$ such that
	$$
	\limsup_n d(z,x_n)\le \limsup_m\limsup_n d(z_m,x_n).
	$$
\end{definition}

Then, Lim-Xu Theorem reads as follows.

\begin{theorem}
Let $X$ be a bounded metric space with uniform normal structure and property (P). Let $T:X\to X$ be a uniformly $k$-Lipschitzian mapping. If $k<\tilde{N}(X)^{-1/2}$, then $T$ has a fixed point. 
\end{theorem}

\begin{remark}
Property (P) can be understood as a replacement in metric spaces of the weakly lower semicontinuity for the type functions in Banach spaces. In fact, property (P) basically gives a formal condition for Theorem \ref{normal} to hold when convex weakly compact domains are no longer available (see Section $5$ where property (P) is also required). It is a general fact that convex weakly compact sets have property (P).
 \end{remark}


\section{Orbit Lipschitzian actions}

With the purpose of obtaining a common fixed point for a family of mappings acting on a same metric space, we  are going to consider actions defined by semigroups. Therefore, we recall some semigroup notions:

A semigroup $\mathcal S$ is a set with an inner product, $\cdot$,  that is associative. Additionally, if the semigroup contains an  element $1\in S$ such that $s\cdot 1=1\cdot s=s$ for all $s\in\mathcal  S$ and for every $s\in \mathcal S$ there is $s^{-1}\in \mathcal S$  with $s\cdot s^{-1}=s^{-1}\cdot s=1$, then $\mathcal S$ is called a group. Trivially, every group is a semigroup. Following  standard semigroup notation (see for instance the monographs  \cite{libro, Na}),  we denote by $\mathcal{S}^1$ the semigroup $\mathcal{S}$ with an identity adjoined, that is $\mathcal{S}^1:=\mathcal{S}\cup\{1\}$ with $s \cdot 1=1\cdot  s=s$ for all $s\in \mathcal{S}$. If $\mathcal{S}$ already has an identity, then $\mathcal{S}^1=\mathcal{S}$.

Every family of mappings generates  a semigroup under the composition operation and  a common fixed point for all the operators of the family is also a common fixed point for all the elements of the corresponding semigroup. Because of that, it is said that a semigroup $\mathcal{S}$ generates an action  over a  space $X$ when there exists a mapping from ${\mathcal S}\times X$ to $X$ such that $(s,x)\to s(x)\in X$, verifying $s\cdot t (x)=s(t(x))$ for all $s,t\in \mathcal S$ and $x\in X$. As it is standard, we will denote $s(x)$ simply by $sx$  and  by $st$ the product $s\cdot t$ for every $s,t\in \mathcal S$ and $x\in X$.

\medskip

In what follows, we will denote by $\mathcal{S}$  a given semigroup and we will denote by $(\mathcal{S},X)$ an action generated by  $\mathcal{S}$ over a metric space $(X,d)$. 
The orbit of $x\in X$ under the action $(\mathcal{S},X)$ is defined by
$$
o(x)=\{x\}\cup \{ sx: s\in\mathcal{S}\}=\{sx: s\in\mathcal{S}^1\}.
$$

\begin{definition}\label{FamUniLips}
An action  $({\mathcal S},X)$ is said to be  $k$-Lipschitzian  for $k>0$  if 
	$$
	d(sx,sy)\le kd(x,y)
	$$
	for all $x,y\in X$ and  $s\in  {\mathcal S}$. \end{definition}

\begin{remark}
Taking into account that a single mapping $T:X\to X$ generates the semigroup $\mathcal{S}=\{T^n,n\ge 1\}$ and the semigroup action $\mathcal{S}\times X\to X$, $(T^n,x)\to T^nx\in X$, the definitions  of $T$ being  uniformly $k$-Lipschitzian and that $\mathcal{S}:=\{T^n:n\ge 1\}$ being a $k$-Lipschitzian  action coincide. 
\end{remark}

When a single mapping $T:X\to X$ is considered, the orbit of $x\in X$ under $T$ is just the orbit of $x$ under the action $\{T^n:n\ge 1\}$.

The following definition was introduced by A. Nicolae in \cite{N} and widely studied by the authors in \cite{EJS} for metric spaces under metric normal structure assumptions:

\begin{definition}
A mapping $T:X\to X$ is said to be  orbit-nonexpansive if 
	$$
	d(Tx,Ty)\le  D(x, o(y))
	$$
	for all $x,y\in X$.
\end{definition}

The following concept naturally 
 extends  the notion of uniformly Lipschitzian mappings and encompasses orbit-nonexpansive operators  as a particular case:

 \begin{definition}\label{uo} A mapping $T:X\to X$ is said to be {\it orbit uniformly  $k$-Lipschitzian } if there is some $k\in (0,+\infty)$ such that
	$$
	d(T^nx,T^ny)\le k D(x, o(y))
	$$
	for all $x,y\in X$ and $n\in \mathbb N$.
	\end{definition}

It is easy to prove that every orbit-nonexpansive mapping is orbit uniformly  $1$-Lipschitzian.
We next show an example of a noncontinuous orbit uniformly  $k$-Lipschitzian mapping. The proof follows the same arguments as those of \cite[Example 2.2]{EJS}.

\begin{example}\label{e3}
Let $0<a<1$, $a\in\mathbb{Q}$ and 
$S_a:[-1,1]\to [-1,1]$ given by
$$
S_a(x)=\left\{
\begin{array}{ll}
\ \ { a x}& \mbox{if $x$ is irrational}\\
{-a x} & \mbox{ if $x$ is rational}\\
\end{array}\right.
$$
Then $S_a$ is orbit uniform  $k$-Lipschitzian  for $k=3a$. \end{example}

Being orbit uniformly $k$-Lipschitzian can also be considered on the action of a semigroup as the next definition states. 

\begin{definition}\label{wrt}
An action $({\mathcal S},X)$   is said to be {\it orbit $k$-Lipschitzian}  if
	$$
	d(sx,sy)\le k D(x,o(y))
	$$
	for all $x,y\in X$ and $s\in \mathcal S$.	
\end{definition}

The following assertions are now straightforward:
\begin{lemma} 
	\begin{enumerate}
	\item Let $T\colon X\to X$ be  uniformly $k$-Lipschitzian. Then $T$ is orbit uniformly  $k$-Lipschitzian.
	
	\item Let $({\mathcal S},X)$ be  a $k$-Lipschitzian action.   Then $(\mathcal{S}, X)$ is orbit $k$-Lipschitzian.
	\end{enumerate}

\end{lemma}





\medskip

 We introduce next one more orbit Lipschitzian condition on a family of mappings that will play a main role in our exposition. 

\begin{definition}\label{general}
	An action $(\mathcal{S},X)$ is said to be {\it strong-orbit  $k$-Lipschitzian } if
	\begin{equation}\label{ostar}
D(sx, o(sy))\le k D(x,o(y)).
\end{equation}
for all $s\in\mathcal{S}$ and $x,y\in X$. 
	 \end{definition}

It is clear that a {\it strong-orbit $k$-Lipschitzian} action is an {\it orbit $k$-Lipschitzian} action too since $d(sx,sy)\le D(sx, o(sy))$ for all $s\in\mathcal{S}$ and $x,y\in X$. We will next check that, for a wide range of semigroups, both notions are in fact equivalent. The next result explores the case of a single mapping. 

\begin{lemma}\label{lemaencompass}
	Assume that  $T\colon X\to X$ is a uniformly $k$-Lipschitzian mapping or, more generally, an orbit uniformly $k$-Lipschitzian mapping. Then the semigroup ${\mathcal S}=\{T^n\colon n\in{\mathbb N}\}$ satisfies the strong-orbit $k$-Lipschitzian condition.
\end{lemma}

\begin{proof}

Fix $n\in\mathbb{N}$. Then for all $m\in\mathbb{N}\cup\{0\}:$
$$
d(T^nx, T^{m+n}y)\le k d(x,T^my)\le k D(x, o(y))
$$
which concludes that for all $m\in\mathbb{N}$:
$$
D(T^n x, o( T^n y))\le k D(x,o(y))
$$ 
for all $(x,y)\in X\times X$.
\end{proof}

Both notions will also be equivalent of some classes of semigroups of mappings.

\begin{lemma}\label{lemaencompass2} 
Assume that $\mathcal{S}$ is a semigroup  with the property that for all 
$s\in\mathcal{S}$ it must be the case that 
$$
\mathcal{S}s\subset s\mathcal{S}.
$$
Then,  if $\mathcal{S}$ is an orbit $k$-Lipschitzian action, it is a strong-orbit $k$-Lipschitzian action too. In particular, this holds  whenever $\mathcal{S}$ is a group or a commutative semigroup. 
\end{lemma}

\begin{proof}	
Let $s\in \mathcal S$ and $x,y\in X$. Take $p\in\mathcal{S}$. Since $ps\in \mathcal{S}s\subset s\mathcal{S}$, there exists $q\in\mathcal{S}$ such that $ps=sq$. Then
$$
d(sx,psy)=d(sx, qsy)\le k D(x, o(sy))\le k D(x,o(y)).
$$	
Taking supremum with respect to $p\in \mathcal{S}^1$, we finally obtain that
$$
D(x, o(sy))\le k D(x,o(y)).
$$
\end{proof}


\begin{remark} The  condition ${\mathcal S}s\subset s{\mathcal S}$ for all $s\in {\mathcal S}$ is verified for  larger families of  semigroups others than commutative ones and groups as, for instance, the left (right) quasi commutative semigroups (a semigroup is called  left (right) quasi commutative if for every $a,b\in\mathcal{S}$, there exists a positive integer $r$ such that $ab=b^ra$ ($ab=ba^r$)). 
Additionally, if $\mathcal{S}$ is a normal semigroup of a group ($t^{-1}st\in \mathcal{S}$ for all $s\in\mathcal{S}$) the above condition also holds. $\mathcal{H}$-commutative semigroups  also  fall within the the scope of Lemma  \ref{lemaencompass2} (see \cite[Chapter 5]{Na}).

\end{remark}

We close this section with a last remark about the boundedness of orbits.

\begin{remark}
In case that a semigroup action $\mathcal{S}$ provides unbounded orbits for every $x\in X$, it is clear that there is no fixed point and that the action provided by $\mathcal{S}$ trivially satisfies all orbit conditions above considered. Nevertheless, in the next two sections we will show common fixed point results as long as adequate conditions are given and there exists an orbit which is bounded. 
\end{remark}

\section{Strong orbital Lipschitzian actions and the Lifschitz characteristic $\kappa(X)$}

In this section we study the existence of common fixed points for strong-orbit $k$-Lipschitzian actions in terms of the Lifschitz characteristic $\kappa(X)$ of the metric space. In fact,  we will consider a weaker Lipschitzian condition on the orbits of the action than this one. This is given by the next result which is the main one in this section.

\begin{theorem}\label{main}
Let $X$ be a complete metric space and $\mathcal{S}$ a semigroup. Assu\-me that $({\mathcal S},X)$ is an action verifying the following condition: There exists $k<\kappa(X)$ such that:
\begin{equation}\label{star}
\begin{array}{c}
	 \mbox{ For  $(x,y)\in X\times X$ with $D(x,o(y))\le D(x,o(x))$ and  $\forall s\in \mathcal S$:}\\
	\\
	\inf_{t\in\mathcal{S}^{1}}D(sx, o(tsy))\le k D(x,o(y)).
\end{array}
\end{equation}
 Then either all the orbits are unbounded  or there exists some $x\in X$ such that $sx=x$ for all $s\in \mathcal S$. 
\end{theorem}

\begin{proof}

Take some $k_0\in (k, \kappa(X))$.

Assume that there is $z\in X$ for which the orbit is bounded.  For $x\in X$, define
$$
r(x):=\inf \{ D(x,o(y)): y\in X \}
$$
which, from our assumption, gives a nonnegative finite number.

Firstly, let us check  that $x\in X$ is a common fixed point for $\mathcal S$  whenever $r(x)=0$. Fix $s\in \mathcal S$. 

Let  $\epsilon>0$ be arbitrary.  We want to show that $d(x,s(x))\le \epsilon$ for any $s\in \mathcal S$. Take $y\in X$ such that $\displaystyle D(x, o(y))\le{ \epsilon\over 1+k_0}$. If   $D(x,o(x))<D(x, o(y))$, since $D(x, o(y))<\epsilon$, we are done.  Otherwise, we can consider condition (\ref{star}) and so there is some $t_0\in\mathcal{S}$ such that 
$$
D(sx, o(t_0sy))\le k_0 D(x,o(y)).
$$
Therefore
$$
\begin{array}{lll}
d(x, sx)&\le & d(x, t_0sy)+d( sx,t_0sy))\le D\left(x,o(y)\right)+ D\left(sx,o(t_0sy)\right)\\
           &  \le & D\left(x,o(y)\right)+k_0 D\left(x,o(y)\right) \\
           & \le &  (1+k_0)D\left(x,o(y)\right) \le \epsilon,
           \end{array}
           $$
as we wanted to prove.

\medskip

We can assume now, with no loss of generality, that there exists $x\in X$ such that $r(x)\in (0,+\infty)$. We will show that there exists $p\in X$ with $r(p)=0$ and so our proof will be complete. Set $r:=r(x)>0$ and consider $\mu,\alpha\in (0,1)$ as in Definition \ref{c}. In particular, $r\le D(x,o(x))$, which implies that
there exists some $s_0\in \mathcal S$ with 
$$
(1-\mu)r< d(x, s_0x).
$$
Take $z\in X$ given by Definition \ref{c}  such that 
\begin{equation}\label{1}
B(x, (1+\mu)r)\cap B(s_0x, k_0 (1+\mu)r)\subset B(z, \alpha r).
\end{equation}
Note that $z$ depends on the point $x$. 
We next prove that $r(z)\le \alpha r$. We split the proof in two cases:

\medskip

{\it Case 1:} $D(x,o(x)) <r(1+\mu)$.

From  condition (\ref{star}), take $t_0\in\mathcal{S}$ verifying 
$$
D(s_0x, o(t_0s_0x))\le k_0 D(x,o(x))<k_0 r(1+\mu).
$$
In particular,  $o(t_0s_0x)\subset o(x)\subset B(x,r(1+\mu))$. Additionally,  the above inequalities imply $o(t_0s_0x)\subset B(s_0x, k_0 (1+\mu)r)$. From $(\ref{1})$ we deduce that $o(t_0s_0x)\subset B(z,\alpha r)$ and $r(z)\le \alpha r$. Additionally, 
$$
d(z,x)\le d(z, t_0s_0x)+d(t_0s_0x,x)\le \alpha r+ r(1+\mu)= Ar
$$
where $A:=\alpha + 1+\mu$. 

\medskip

We next prove that  $D(z,o(z))<(1+k_0)\alpha r$:

\medskip

Again we consider two subcases:

{\it Case 1.1:}  Assume $D(z,o(z))\le D(z,o(t_0s_0x))$, which directly implies that
$$D(z,o(z))<\alpha r<(1+k_0)\alpha r.
$$

{\it Case 1.2:}  Assume otherwise that $D(z,o(t_0s_0x))\le D(z,o(z))$. Then fix  any $t\in \mathcal{S}$. According to condition (\ref{star}) applied to the points $z$ and $t_0s_0 x$, there is some $t_1\in\mathcal{S}$ with 
$$
D(tz, o(t_1 tt_0s_0x))\le k_0 D(z, o(t_0s_0x)).
$$
 Therefore,
$$
\begin{array}{lll}
d(z,tz) &\le & d(z, t_1 t t_0 s_0 x)+ d( tz, t_1 t t_0 s_0 x)\\
             &\le & D(z, o(t_0s_0x))+ D(tz, o(t_1 t t_0 s_0x))\\
           & \le & D(z, o(t_0s_0x))+ k_0 D(z, o(t_0s_0x))\\
           &\le & (1+k_0)\alpha r.\\
           \end{array}
           $$
Taking supremum when $t\in \mathcal{S}$, we finally deduce
\begin{equation}\label{xy}
D(z, o(z))\le (1+k_0)\alpha r.
\end{equation}

\medskip 

{\it Case 2:} Assume now that {\it Case 1} fails so $r(1+\mu)\le  D(x,o(x)).$ Now we can select some $y\in X$ verifying 
$$
D(x,o(y))<r(1+\mu)\le D(x,o(x)).
$$
According to condition (\ref{star}),  we can choose $t_1\in\mathcal{S}$ such that
$$
D(s_0x, o(t_1 s_0y))\le k_0 D(x, o(y))\le k_0(1+\mu)r.
$$
Note that $ o(t_1 s_0y)\subset o(y)\subset B(x, (1+\mu)r)$ and
$$
o(t_1 s_0y)\subset B(s_0x, k_0 (1+\mu)r).
$$
Together with (\ref{1}), we deduce that $o(t_1s_0y)\subset B(z,\alpha r)$ and
$$
r(z)\le \alpha r.
$$
Besides,
$$
d(x,z)\le d(x,t_1 s_0y)+d(t_1 s_0 y,z)\le Ar.
$$
Likewise, replacing $x$ by $y$ in a previous argument,  we can  deduce that $D(z, o(z))\le (1+k)\alpha r.$

In any case, we have found $z\in X$ such that $r(z)\le \alpha r$, $d(x,z)\le A r$ and $D(z, o(z))\le (1+k)\alpha r.$

From this moment on, we will proceed following some standard arguments in order to find a Cauchy sequence whose limit $p$  verifies  $r(p)=0$ and therefore is necessarily a common fixed point for the action. Let us call $x_0:=x$ and take $z=z(x)$, the point $z$ depending on $x$ as above. Then, we define $x_1:=z$.  If $r(x_1)=0$ we make $x_1=p$ and we are done, otherwise we can reason as above and choose $x_2=z(x_1)$. Following this inductive argument, we construct a sequence $(x_n)\subset X$ verifying the following properties: 
 
 \begin{itemize}
 
 \item[i)] $r(x_{n})\le \alpha r(x_{n-1})$

 \item[ii)] $d(x_{n},x_{n-1})\le A r(x_{n-1})$.

 \item[iii)] $D(x_{n},o(x_{n}))\le (1+k)\alpha r(x_{n-1})$
 
 \end{itemize}
for all $n\in\mathbb{N}$.

Properties i) and ii) imply that $(x_n)$ is a Cauchy sequence in $X$ and, by completeness, there is $p\in X$ such that $\lim_n d(x_n,p)=0$.  Note that
$$
D(p,o(x_n))\le d(p,x_n)+ D(x_n, o(x_n)).
$$
Therefore $D(p,o(x_n))$ tends to $0$ as $n$ goes to infinity.  This finally shows that $r(p)=0$ and so $p$ is a common fixed point for the action $\mathcal S$ on $X$, which completes the proof. 
\end{proof}

Since it is direct to check that being a strong-orbit $k$-Lipschitzian action implies that the action satisfies condition (\ref{star}), the next corollary is immediate.
	
\begin{corollary}
	Let $X$ be a complete metric space and $\mathcal{S}$ a semigroup. Assu\-me that $({\mathcal S},X)$ is a strong-orbit $k$-Lipschitzian  action with $k<\kappa (X)$.
	Then either all the orbits are unbounded  or there exists some $x\in X$ such that $sx=x$ for all $s\in \mathcal S$. 
\end{corollary}

The next example compares some of the Lipschitzian conditions studied in this work.

\begin{example} For $X=[0,1]$ define the mapping $T(x)=x^2$ if $x\in [0,1)$ and $T(1)=0$. We prove that the semigroup $\{T^n: n\in\mathbb{N}\}$ satisfies  condition (\ref{star}) for $k=1$ while it fails to be a Lipschitzian and orbit uniformly $k$-Lipschitzian action for any $k<2$.
\end{example}

\begin{proof}
Notice  that for all $x\in [0,1]$, $0\in\overline{o(x)}$. This implies that $D(x, o(x))=x$ and $D(x,o(y))=\max\{x, y-x\}$.

Let $y\in [0,1)$ with $D(x,o(y))\le D(x,o(x))$. This implies that $D(x,o(y))=\max\{x, y-x\}\le D(x,o(x))=x$ and $D(x,o(y))=x$.  
Fix $n\in\mathbb{N}$, 
$$
\begin{array}{lll}
\inf_{m\in\mathbb{N}} D(T^n x, o(T^m T^n y))&=&\inf_{m\ge n} D(T^n x, o( T^m y))  \\
& = &\inf_{m\ge n} \max\{x^{2n}, y^{2m}-x^{2n}\}\\
&\le & x^{2n}\le x=D(x, o(y)).\\
\end{array}
$$
In case that $y=1$, $T^m y=0$ for all $m\ge 2$ and the previous argument also follows.  Therefore condition (\ref{star}) is satisfied for $k=1$.

It is straightforward to show that $T$ is not uniform Lipschitzian for any $k$ since $T$ is not continuous. Additionally $\{T^n: n\in\mathbb{N}\}$ is not orbit $k$-Lipschitz for any $k\in (0,2)$ either. 

Indeed, consider $0<y<1$ and $x=y/2$. In this case $D(x,o(y))=y-x=d(x,y)=y/2$. 
 Additionally 
 $$
 {d(T^n(y/2),T^n y)\over d(y/2,y)}={y^{2n}-\left({y\over 2}\right)^{2n}\over {y\over 2}}=2 y^{2n-1}\left(1-{1\over 2^{2n}}\right).
  $$
 Considering for instance $y=e^{-1/n^2}$, it is not difficult to check that 
  $$
  \sup\left\{{d(T^n(y/2),T^n y)\over d(y/2,y)}: n\in\ \mathbb{N}, y\in [0,1)\right\}=2,
  $$
 from where it follows that the semigroup generated by $T$ is not an orbit uniformly $k$-Lipschitzian action for any $k<2$. \end{proof}


In a case of a single mapping, we deduce the following:

\begin{corollary}\label{mainT}
Let $X$ be a complete metric space and $T:X\to X$ a mapping such that there exists some  $k<\kappa(X)$ for which the following holds:
\begin{equation}\label{starT}
\begin{array}{c}
	 \mbox{ For  $(x,y)\in X\times X$ with $D(x,o(y))\le D(x,o(x))$:}\\
	\\
	\lim_m D(T^nx, o(T^m y))\le k D(x,o(y))
\end{array} 
\end{equation}
for all $n\in\mathbb{N}$.
Then either all the orbits are unbounded  or there exists some $x\in X$ such that $Tx=x$.  
\end{corollary}

Since $\kappa(X)\ge 1$  for any metric space $X$, the above corollary implies the existence of a fixed point in a complete metric space as long as condition (\ref{starT}) holds for  some $k<1$. The following  is an extension of Lifschitz Theorem (Theorem \ref{Li} above):

\begin{corollary}
	Let $X$ be a complete metric space and $T\colon X\to X$ such that there exists
	$k<\kappa (X)$ with
	$$
	d(T^n x,T^n y)\le k D(x,o(y))
	$$
	for all $n\in\mathbb{N}$ and for all $(x,y)\in X\times X$ with $D(x,o(y))\le D(x,o(x))$. If some orbit is bounded, then $T$ has a fixed point.  
\end{corollary}


\begin{remark} Note that Theorem \ref{main} is optimal since the following example shows that the Lipschitz constant cannot be improved.  Indeed,   consider the semigroup generated by the mapping defined on $[0,1]$ given by 
	$$
	T(x)=\begin{cases}
	1, & \text{if $x\in [0,1)$,}\\
	0, & \text{if $x=1$.}
	\end{cases}
	$$
To check the orbit Lipschitzian character of $T$, the worst case happens for $x=1/2$ and $y=1$. In this case, we have that $o(y)=\{1,0\}$, and so $D(x, o(y))=1/2$. From where, $d(Tx,Ty)=2D(x,o(y))$. As a consequence, the semigroup generated by $\{T\}$ enjoys the orbit $2$-Lipschitzian condition (and so condition (\ref{star}) in Theorem \ref{main} too), while it fails to have a fixed point. Remember that $\kappa ([0,1])=2$. 
\end{remark}


\begin{example}
Let $X=B_{\ell_p}$ be the closed unit ball of $\ell_p$ endowed with the usual metric provided by the norm $\V\cdot\V_p$ for $1\le p<+\infty$. Let us denote by $\{e_n\}$ the standard basic vectors and define $T:B_{\ell_p}\to B_{\ell_p}$ as $T(x)=e_1$ if $x\ne e_n$ for all $n\in\mathbb{N}$, and $T(e_n)=e_{n+1}$. It is clear that $T$ lacks of any fixed point and so does the semigroup generated by $\{T^n: n\in\mathbb{N}\}$. We next prove that the action satisfies the strong-orbit $2^{1/p}$-Lipschitzian condition.

Indeed, for every $x\in B_{\ell_p}$ we have that $\limsup_n \V x-e_n\V=(\V x\V_p+1)^{1/p}$. Moreover, if $x\ne x_n$ for all $n\in\mathbb{N}$, $o(x)=\{x,e_1,e_2,\cdots \}$ and $o(Tx)=\{e_1,e_2,\cdots \}$. In particular, setting $x:=0$, $D(x, o(x))=1$ and $D(x, o(y))=1$ for all $y\in B_{\ell_p}$. Additionally, for every $n\in\mathbb{N}$, $y\ne 0$
$$
D(T0, o(T^n y))=D(e_1, o(T^n y))=2^{1\over p}=2^{1\over p} D(0, o(y))
$$
which shows that the semigroup satisfies the  strong-orbit $2^{1/p}$-Lipschitz condition and its extension considered in Theorem \ref{main} given by condition (\ref{star}). The above shows in particular that $\kappa(\ell_p)\ge 2^{1/p}$ for $1\le p<+\infty$ and from the example included in Remark \ref{ExP} (see Section 5), $\kappa(\ell_\infty)=1$.

\end{example}

\section{Strong-orbit Lipschitzian actions and   uniform normal structure}

%
%

The main goal of this section is to obtain a common fixed point result for strong-orbit Lipschitzian actions in metric spaces satisfying the (metric) uniform normal structure (see Section $2$ for definitions) and, therefore, to  achieve  extensions of Casini-Maluta Theorem \cite{CM}, Lim-Xu Theorem \cite{LX} and some others published in the literature.

\medskip

 Given a semigroup $\mathcal{S}$, a preorder can be defined as follows: 
$$
 s\le t \text{ if } t\in \mathcal{S}^1s=\{s\}\cup \mathcal{S} s.
$$
Notice that $\le$ is a reflexive and transitive  relation, but it need not be antisymmetric.  In fact, when $\mathcal{S}$ is a group, $t\in \mathcal{S}s$ for every $t,s\in\mathcal{S}$.

Along this section we assume that the preorder is total, that is, for every $s,t\in \mathcal{S}$,  either $s\ge t$ or $t\ge s$. In particular, this implies that $(\mathcal{S},\le)$ is a directed set.
Particular examples are groups, $(\mathbb{N},+)$, $((1,+\infty), +)$, $((1,+\infty), \cdot)$.

\medskip

Property (P) has already been recalled in Section $2$ for sequences (see Definition \ref{Psequences}), we will need it for nets in this section.

\begin{definition}\label{Pnets}
	A metric space $X$ is said to have property (P) for nets if given  two bounded nets $\{ x_s\}_s$ and $\{ z_s\}_s$ in $X$, such that $z_s\in  {\rm cov}(\{x_j: j\ge s\})$, there is $z\in \bigcap_s {\rm cov}(\{ z_j\; :\; j\geq s\})$ such that
	$$
	\limsup_s d(z,x_s)\le \limsup_t\limsup_s d(z_t,x_s).
	$$
	
\end{definition}

Recall that a metric space is said to be {\it proper} if its closed balls are compact. Then the following lemma is easy to prove.

\begin{lemma} 
	Every proper metric space has property (P) for nets. \end{lemma}

\begin{proof} Let $\{x_s\}$, $\{z_s\}$ two bounded nets as in Definition \ref{Pnets}. We can assume that they lie in a compact closed ball. Let  $z$ be a cluster point of the net $\{z_s\}$ for the metric topology. Then $z\in \bigcap_s {\rm cov}(\{ z_j\; :\; j\geq s\})$ and there exists a subnet $\{z_{s_\alpha}\}$ such that $\lim_\alpha d(z, z_{s_\alpha})=0$. Thus,
$$
\begin{array}{lll}
\limsup_s d(z,x_s)& \le & \lim_\alpha d(z, z_{t_\alpha})+\liminf_\alpha \limsup_s  d(z_{t_\alpha}, x_s)\\
                              &\le & \limsup_t \limsup_s  d(z_{t}, x_s).\\
                    \end{array}
                    $$\end{proof}

Although we are strongly inspired by the arguments used in \cite{CM} and \cite{LX}, we will proceed introducing all the details for the sake of a better understanding.

\begin{remark}\label{min} Note that  if $A$ is a subset of $X$ and $z\in {\rm cov}(A)$ then $d(y,z)\le D(y,A)$ for all $y\in X$. 
This, in particular, implies that if $(A_s)_s$ is a net of  subsets  of $X$ such that  $z\in\bigcap_s {\rm cov}(A_s)$, then, for all $y\in X$, it must be the case that
$$
d(z,y)\le \inf_s D(y,A_s).
$$
Additionally, $\delta(A)=\delta({\rm cov}(A))$ for all $A\subset X$.
\end{remark}

\begin{theorem}\label{teor:Nwrt}
	Let $X$ be a complete  metric space with  (metric) uniform normal structure and property (P) for nets. Let $(\mathcal{S},\le)$ be a totally preordered semigroup such that $(\mathcal{S},X)$ is a strong-orbit $k$-Lipschitzian action. If $k< \tilde{N}(X)^{-1/2}$ and some orbit is bounded,  there is $z\in X$ such that $sz=z$ for all $s\in \mathcal S$.

\end{theorem}

\begin{proof}

We will start the proof by making some preliminary observations.  Let $x\in X$:

\begin{itemize}

\item[i)] If $s,t\in\mathcal{S}$,  $s\le t$ then   $o(tx)\subset o(sx)$  and $o(sx)=\{ ux: u\ge s\}$.

\item[ii)]  $\{ sx: s\in\mathcal{S}\}$ is a net in $X$ and $\{o(sx): s\in \mathcal{S}\}$ forms a nonincreasing net of subsets of $X$. Additionally for all $y\in X$:
$$
\limsup_s d(y , sx)=\lim_s D(y, o(sx))=\inf_s D(y, o(sx)).
$$

\end{itemize}

We choose a constant $c$ such that $1>c>\tilde{N}(X)$ and $k< c^{-1/2}$. Fix $x_0\in X$ for which the orbit is bounded and consider the net $\{sx_0: s\in\mathcal{S}\}$. From the definition of the uniform normal structure coefficient  $\tilde{N}(X)$, for $s\in \mathcal S$, there exists 
$$
z_s\in {\rm cov}(o(sx_0))\subset B(z_s, c \delta(o(sx_0))).
	$$
	Applying the condition (P) to the nets $\{sx_0: s\in\mathcal{S}\}$ and $\{z_s:s\in\mathcal{S}\}$, there exists $z_0\in \bigcap_s {\rm cov}(o(sx_0))$ such that:
	$$
	\limsup_s d(z_0, sx_0)\le \limsup_t\limsup_s d(z_t, sx_0).
	$$
Writing the above inequality in terms of the orbits: 
	$$
	\begin{array}{lll}
	\lim_s D(z_0, o(sx_0))&\le & \limsup_t \lim_s D(z_t, o(sx_0))=\limsup_t \inf_s D(z_t, o(sx_0))\\
	& \le& \limsup_tD(z_t, o(tx_0))\le c \limsup_t \delta(o(tx_0))\\
	&\le & c \delta (o(x_0)).\\
	\end{array}
	$$
	Thus, if we call $x_1:= z_0$ we have just prove that
	$$
	x_1\in \bigcap_s {\rm cov}(o(sx_0))
	$$
	 and, by Remark \ref{min},
	$$
	\lim_s D(x_1, o(sx_0)) \le c  \delta(o(x_0)).
	$$
	Repeating the process, we can find a sequence $(x_j)\subset X$ such that:
	\begin{equation}\label{int}
	x_{j+1}\in \bigcap_s {\rm cov}(o(sx_j))
	\end{equation}
	 and, again by Remark \ref{min},
	\begin{equation}\label{ine}
	\lim_s D(x_{j+1}, o(sx_j))\le  c\delta(o(x_j)).
	\end{equation}
We next find an upper estimation for $\delta(o(x_j))$: 

Let $s,t\in \mathcal{S}$ with $t\ne s$. Since the preorder is total,  we can assume that $s\geq t$ which implies that $s=at$ for some $a\in\mathcal{S}$. Then
$$
d(sx_j, tx_j)  =  d(atx_j, tx_j)\le D(tx_j, o(tx_j))\le kD(x_j, o(x_j))
$$
The above implies that 
\begin{equation}\label{d}
\delta(o(x_j))\le 	kD(x_j, o(x_j)).
\end{equation}
Additionally, using (\ref{int}) and Remark \ref{min},  for every $s\in \mathcal{S}$:
$$
\begin{array}{lll}
d(x_j, sx_j)& \le &  \inf_{t\in\mathcal{S}} D(sx_j, o(t x_{j-1}))\le \inf_{t\in\mathcal{S}} D(sx_j, o(st x_{j-1}))\\
             & \le & \inf_{t\in\mathcal{S}} k D(x_j, o(t x_{j-1})).\\
             \end{array}
$$
Taking supremum in $s\in\mathcal{S}$, we deduce that 	
\begin{equation}\label{j}
D(x_j, o(x_j))\le k	\inf_{t\in\mathcal{S}}  D(x_j, o(t x_{j-1}))= k \lim_t D(x_j, o(t x_{j-1})).
\end{equation}

For every $j\ge 0$, we define $D_j:= \lim_s D(x_{j+1}, o(sx_j))$. Considering (\ref{ine}), (\ref{d}) and (\ref{j}), we obtain that: 
\begin{equation}\label{r}
D_j\le ck^2 D_{j-1},
\end{equation}
for every $j\in\mathbb{N}$.
Set $r:=ck^2<1$. In particular $\lim_j D_j=0$ and $\lim_j D(x_j, o(x_j))=0$ by (\ref{j}).

\medskip

We next check that $(x_j)$ is a Cauchy sequence by using  (\ref{int}), Remark \ref{min},  (\ref{j}) and (\ref{r}):
$$
\begin{array}{lll}
d(x_{j+1}, x_j)&\le &\inf_{s\in\mathcal{S}} D(x_j, o(sx_{j}))\le D(x_j, o(x_j))
\\                    &\le & k D_{j-1}\le k rD_{j-2}\le k r^{j-1} D_0.
\end{array}
$$
Set $w=\lim_j x_j$,  which exists by the completeness of the metric space. We will prove that $w$ is a common fixed point for $\mathcal{S}$. Firstly, let us check that $\lim_j D(w, o(x_j))=0$.
Indeed, let $t\in\mathcal{S}$:
$$
d(w, tx_j)\le d(w, x_j)+ d(x_j, tx_j)\le d(w, x_j)+D(x_j, o(x_j)).
$$
The above implies that 
$$
D(w, o(x_j))\le d(w, x_j)+ D(x_j, o(x_j))\to_j 0. 
$$
Finally, given $s\in\mathcal{S}$:
$$
\begin{array}{lll}
d(w, sw)& \le & d(w, x_j)+ d(x_j, sx_j)+ d(sx_j, sw)\\
            & \le & d(w, x_j)+ D(x_j, o(x_j)) + D(s w, o(sx_j))\\
            & \le & d(w, x_j)+ D(x_j, o(x_j))+ k D(w, o(x_j))\to_j 0. \\
            \end{array}
            $$
Therefore the proof is concluded. 

	\end{proof}

Notice that when $\mathcal{S}$ is countable,  it is enough to consider property (P) for sequences.
When $\mathcal{S}$ is additionally a commutative semigroup, we have the following corollary: 

\begin{corollary} Let $X$ be a metric space with uniform normal structure and property (P) for nets. 
Let $\mathcal{S}$ be a commutative semigroup, or more generally, such that $\mathcal{S}s\subset s\mathcal{S} $ for all $s\in\mathcal{S}$. Assume that $(\mathcal{S},\le)$ is total. Then every $(\mathcal{S},X)$ action which is orbit $k$-Lipschitzian with $k<\tilde{N}(X)^{-1/2}$ has a common fixed point. 
\end{corollary}

For a single mapping $T$, we have the next result.

\begin{corollary} Let $X$ be a metric space with uniform normal structure and property (P) for sequences. 
Let $T:X\to X$ be orbit uniform $k$-Lipschitzian with  $k<\tilde{N}(X)^{-1/2}$. Then $T$ has a fixed point.  
\end{corollary}
\begin{proof}
	It is enough to apply Theorem \ref{teor:Nwrt} to the action generated by $T$ and its iterates.
\end{proof}

\begin{remark}
Given $\mathcal{S}$ a semigroup, if we fix some $t\in\mathcal{S}$ and consider 
$$
\mathcal{S}_t:=S^{1}t= \{t\}\cup \mathcal{S}t=\{s\in\mathcal{S}: s\ge t\},
$$
It turns out that $\mathcal{S}_t$ is also a semigroup. Assume that $\mathcal{S}$ is an action over a metric space $X$ and all the mappings in $\mathcal{S}_t$ share a common fixed point $w\in X$. Then $w$ is a common fixed point for all $s\in\mathcal{S}$. Indeed $st(w)=w=t(w)$ so $s(w)=w$ for all $s\in\mathcal{S}$. This implies in particular that in order to obtain a fixed point for a semigroup action, it is enough to verify the required hypotheses  for some $\mathcal{S}_t$ with $t\in\mathcal{S}$. In particular, if $\mathcal{S}$ is generated by a single mapping $T$, the existence of a common fixed point for  $\mathcal{S}_{n_0}=\{T^n: n\ge n_0\}$  for some  $n_0\in\mathbb{N}$, implies the existence of a fixed point for $T$. 

\end{remark}

\begin{remark}\label{ExP}

We would like to  point out the fact that only one bounded orbit is needed for our results in contrast to the fact that bounded metric spaces are usually required in the literature
(see for instance in \cite{CM, DuTa, LX}).  This is somehow surprising if we consider the following example, given by S. Prus (see for instance \cite[p. 412]{Handb}: Let $T:\ell_\infty\to \ell_\infty$ be defined by 
$$
T(x)=(1+\limsup_n x_n, x_1, x_2, x_3,\cdots,....).
$$
It is very easy to check that $T$ is an isometry fixed point free mapping, where all the orbits are bounded. Additionally, $\ell_\infty$ is a hyperconvex metric space, which implies that $\displaystyle \tilde{N}(\ell_\infty)=\frac{1}{2}$ (see \cite{EK} for details). Therefore,   if $\ell_\infty$ verified property $(P)$ for sequences, $T$ would have to have a fixed point. In conclusion, $\ell_\infty$ does not verify property $(P)$ for sequences. 

\medskip

In case that $\mathbb{R}^n$ is endowed with the supremum norm and $d(x,y):=\V x-y\V_\infty$ for $x,y\in \mathbb{R}^n$,   we know from hyperconvexity that $\tilde{N}(\mathbb{R}^n)=1/2$ and that property (P) holds since $\mathbb{R}^n$ is  a  proper metric space (note that in this case $\kappa(\mathbb{R}^n)=1$ when  $n>1$). Thus, Theorem \ref{teor:Nwrt} can be applied and we can deduce the existence of a common fixed point for a strong-orbit $k$-Lipschitzian action  or a fixed point for an orbit uniformly $k$-Lipschitzian mapping, when $k<\sqrt{2}$. Besides,  Theorem  \ref{teor:Nwrt} can also be applied  to the Isbell's hyperconvex hull of every compact metric space (see   for instance \cite[Section 8]{EK}), since Isbell's hyperconvex hull of a compact metric space is hyperconvex and  compact and therefore property (P) holds.\end{remark}


\medskip





\begin{thebibliography}{99}

\bibitem{AFH} A. Amini-Harandi, M. Fakhar,  H. R. Hajisharifi, {\it Weak fixed point property for nonexpansive mappings with respect to orbits
in Banach spaces}. J. Fixed Point Theory Appl. 18, 2016, 601-607.

\bibitem{libro}J.F. Berglund, H.D. Junghem,  P. Milnes.    Analysis on Semigroups: Function Spaces, Compactifications, Representations. Canadian Mathematical Society Series of Monographs and Advance Texts. John Wiley and Sons, Inc, 1989.


\bibitem{CM} E. Casini, E. Maluta, {\it Fixed points of uniformly Lipschitzian mappings in spaces with uniformly normal structure}. Nonlinear Anal., 9, 1985, 103-108.

\bibitem{DKS} S. Dhompongsa, W.A. Kirk, B. Sims, {\it Fixed points of uniformly Lipschitzian mappings}. Nonlinear Anal. 65, 2006, 762-772.

\bibitem{DR} D. J. Downing, W. 0. Ray, {\it Uniformly Lipschitzian semigroup in Hilbert space}.
Canad. Math. Bull. 25, 1982, 210-214.

\bibitem{BM} M. S. Drodski\u{i}, D. P. Milman, {\it On the center of a convex set (Russian)}, Doklady Akad. Nauk SSSR (N.S.) 59, 1948, 837-840.

\bibitem{DuTa} L. A. Dung, D. H. Tan, {\it Fixed points of semigroups of Lipschitzian mappings}. Acta Mathematica Vietnamica, 28, 2003, 89-100.

\bibitem{DT}  D. Downing, B. Turett. {\it Some properties of the characteristic of convexity relating to fixed point theory.} Pacific J. Math. 104, 1983, 343-50.

\bibitem{EJS0} R. Esp\'{\i}nola, M.  Jap\'on, D. Souza. {\it New examples of subsets of c with the FPP and stability of the FPP in hyperconvex spaces.} J. Fixed Point Theory Appl. 23 (2021), no. 3, Paper No. 45, 19 pp.

\bibitem{EJS} R. Esp\'{\i}nola, M. Jap\'on, D. Souza, {\it Fixed Points and Common Fixed Points for Orbit-Nonexpansive Mappings in Metric Spaces}. Mediterr. J. Math. (2023), 20:182. 

\bibitem{EK} R. Esp\'{\i}nola Garc\'{\i}a, A. Khamsi. Introduction to hyperconvex spaces. Handbook of Metric Fixed Point Theory, Chapter 13. Kluwer Academic Publishers, 2001. 



\bibitem{GK} K. Goebel, W. A. Kirk. Topics in metric fixed point theory. Cambridge University Press, 1990. 

\bibitem{GK1} K. Goebel, W. A. Kirk, {\it A fixed point theorem for transformations whose iterates have uniform Lipschitz constant}. Studia Math. 47, 1973, 135-140.

\bibitem{Go} M. G\'ornicki, {\it Fixed points of Lipschitzian semigroups in Banach spaces}. Studia Math. 126, 1997, 101-113.



\bibitem{HH} Y. Y. Huang, C. C. Hong, {\it Common fixed point theorems for semigroups on metric spaces}. Internat. J. Math. and Math. Sci. 22, 1999, 377-386.

\bibitem{Ishi} H. Ishihara, W. Takahashi, {\it Fixed point theorems for uniformly Lipschitzian semigroup in Hilbert spaces.} J. Math. Anal. Appl. 127, 1987, 206-210.





\bibitem{KK} M. A. Khamsi, W. A. Kirk, An introduction to metric spaces and fixed point theory, John Wiley and Sons, 2001.



\bibitem{Handb} W.A. Kirk, B. Sims. Handbook of metric fixed point theory, Kluwer Academic Publishers, Dordrecht, 2001. 

\bibitem{L}  E. A. Lifschitz, {\it Fixed point theorems for operators in strongly convex spaces}, Voronez
Gos. Univ. Trudy Mat. Fak., 16, 1975, 23-28. (Russian)

\bibitem{LX} T. C. Lim, H. K. Xu, {\it Uniformly Lispchitzian mappings in metric spaces with uniform normal structure}. Nonlinear Anal. 25, 1995, 1231-1235. 

\bibitem{Na} A. Naggy. Special Classes of Semigroups. Advances in Mathematics. Springer New York, 2001.


\bibitem{N}  A. Nicolae, {\it Generalized asymptotic pointwise contractions and nonexpansive mappings involving orbits}. Fixed Point
Theory Appl. 2010 458265.

\bibitem{TaXu} K. K. Tan, H. K. Xu, {\it Fixed point theorems for Lipschitzian semigroups in Banach spaces}. Nonlinear Anal. 20, 1993, 395-404.

\bibitem{W} A. Wi\'snicki, {\it H\"older continuous retractions and amenable semigroups of uniformly Lipschitzian mappings in Hilbert spaces}. Topol. Methods Nonlinear Anal. 43, 2014, 89–96. 

\bibitem{WW} A. Wi\'snicki, J. Wo\'sko. {\it Uniformly Lipschitzian group actions on hyperconvex spaces.}   Proc. Amer. Math. Soc. 114 (9), 2016, 3813-3824.

\end{thebibliography}
\end{document}